\documentclass[reqno,a4paper,12pt]{amsart} 

\usepackage{amsmath,amscd,amsfonts,amssymb}
\usepackage{mathrsfs,dsfont}

\numberwithin{equation}{section}
\numberwithin{figure}{section}

\newcommand\R{\mathbb{R}}

\newcommand\Z{\mathbb{Z}}
\newcommand\F{\mathcal{F}}

\newcommand\T{\mathbb{T}}
\newcommand\gam{\gamma}

\newcommand\lam{\lambda}
\newcommand\Lam{\Lambda}

\newcommand\sig{\sigma}
\newcommand\Om{\Omega}

\newcommand\1{\mathds{1}}
\newcommand\eps{\varepsilon}

\renewcommand\le{\leqslant}

\renewcommand\leq{\leqslant}
\renewcommand\geq{\geqslant}
\newcommand\sbt{\subset}

\renewcommand\hat{\widehat}
\renewcommand\Re{\operatorname{Re}}

\newcommand{\ft}[1]{\widehat{#1}}
\newcommand{\dotprod}[2]{\langle #1 , #2 \rangle}

\newcommand{\supp}{\operatorname{supp}}

\newcommand{\zeros}{Z}

\newcommand{\zfft}{Z(\ft{f}\,)}

\newcommand{\half}{\tfrac{1}{2}}

\theoremstyle{plain}
\newtheorem{thm}{Theorem}[section]
\newtheorem{lem}[thm]{Lemma}
\newtheorem{corollary}[thm]{Corollary}

\newtheorem*{claim*}{Claim}

\newcommand{\thmref}[1]{Theorem~\ref{#1}}
\newcommand{\secref}[1]{Section~\ref{#1}}

\newcommand{\lemref}[1]{Lemma~\ref{#1}}

\newcommand{\corref}[1]{Corollary~\ref{#1}}

\theoremstyle{definition}

\newtheorem*{definition*}{Definition}
\newtheorem*{remarks*}{Remarks}
\newtheorem*{remark*}{Remark}

\newenvironment{enumerate-roman}
{\begin{enumerate}
\addtolength{\itemsep}{5pt}
}
{\end{enumerate}}

\newenvironment{enumerate-alph}
{\begin{enumerate}
\addtolength{\itemsep}{5pt}
}
{\end{enumerate}}

\newenvironment{enumerate-num}
{\begin{enumerate}
\addtolength{\itemsep}{5pt}
}
{\end{enumerate}}

\newenvironment{enumerate-text}
{\begin{enumerate}
\addtolength{\itemsep}{5pt}
}
{\end{enumerate}}

\begin{document}

\title[An example concerning Fourier analytic criteria for tiling]
{An example concerning Fourier analytic criteria for translational tiling}

\author{Nir Lev}
\address{Department of Mathematics, Bar-Ilan University, Ramat-Gan 5290002, Israel}
\email{levnir@math.biu.ac.il}

\date{September 27, 2021}
\thanks{Research supported by ISF Grant No.\ 227/17 and ERC Starting Grant No.\ 713927.}

\begin{abstract}
It is well-known that the functions $f \in L^1(\mathbb{R}^d)$ whose translates along a lattice $\Lambda$ form a tiling, can be completely characterized in terms of the zero set of their Fourier transform. We construct an example of a discrete set $\Lambda \subset \mathbb{R}$ (a small perturbation of the integers) for which no characterization of this kind is possible: there are two functions $f, g \in L^1(\mathbb{R})$ whose Fourier transforms have the same set of zeros, but such that $f + \Lambda$ is a tiling while $g + \Lambda$ is not. 
\end{abstract}

\maketitle


\section{Introduction}

\subsection{} 
Let $f$ be a function in $L^1(\R)$ and let
$\Lam \sbt \R$ be a discrete set. 
We say that \emph{$f$ tiles $\R$
at level $w$} with the translation set $\Lam$,
or that 
\emph{$f+\Lam$ is a tiling of $\R$
at level $w$} (where $w$ is a constant), if 
\begin{equation}
\label{eqI1.1}
\sum_{\lambda\in\Lambda}f(x-\lambda)=w\quad\text{a.e.}
\end{equation}
and the series in \eqref{eqI1.1} converges absolutely a.e.

In the same way one can define tiling of $\R^d$
by translates of a function $f \in L^1(\R^d)$.

For example, if $f = \1_\Omega$ 
is the indicator function of a set $\Omega$,
and $f + \Lam$ is a tiling at level $1$, then this means
that the translated copies $\Omega+\lam$, $\lam\in\Lam$,
fill the whole space without overlaps up to measure zero. 
To the contrary, for tiling by a
 general real or complex-valued function
$f$, the translated copies may have 
overlapping supports.

Tilings by translates of a function 
have been studied by several authors, see, in particular, \cite{LM91}, 
\cite{KL96}, \cite{Kol04}, \cite{KL16}, \cite{Liu18}, \cite{KL21}.

\subsection{}
It is well-known that in the study of translational
tilings, the set 
\begin{equation}
\label{eqI5.1}
 \zfft := \{ t  : \ft{f}(t) = 0\}
\end{equation}
of the zeros of the Fourier transform
\begin{equation}
\label{eqI5.2}
\ft{f}(t)=\int f(x) \exp(-2 \pi i tx) dx
\end{equation}
plays an important role.  For example, 
let $\Lam$ be  a lattice in $\R^d$,
then $f + \Lam$ is a tiling 
if and only if the set $\zfft$
contains $\Lam^* \setminus \{0\}$, where $\Lam^*$ is
the dual lattice. This means that the functions $f$ that tile
by a lattice $\Lam$ can be  completely characterized
in terms of the zero set $\zfft$.
(One can show that  the tiling 
level is given by $w = \ft{f}(0) \det(\Lam)^{-1}$.)

The  necessity of the condition for tiling 
in the last example can be generalized as follows.
For a discrete set $\Lambda \sbt \mathbb{R}$
 we consider the measure
\begin{equation}
\label{eqI5.4}
\delta_\Lambda:=\sum_{\lambda\in\Lambda}\delta_\lambda.
\end{equation}
We will assume that $\Lambda$ has \emph{bounded density}, which means that
 \begin{equation}
\label{eqI5.5}
 \sup_{x\in\mathbb R} \#(\Lambda\cap[x,x+1))<+\infty.
 \end{equation}
In particular \eqref{eqI5.5} implies that the measure $\delta_\Lambda$ is a temperate distribution on $\mathbb R$, so it has
a well-defined Fourier transform 
 $\hat\delta_\Lambda$
  in the distributional sense.

\begin{thm}[{\cite{KL16}}]
\label{thmI3.1}
Let $f\in L^1(\mathbb R)$, and $\Lambda\sbt\R$
 be a discrete set of bounded density. If $f+\Lambda$ is a tiling
at some level $w$, then
\begin{equation}
\label{eqI3.1}
\supp(\hat\delta_\Lambda) \setminus \{0\} \subset \zfft.
\end{equation}
\end{thm}

A similar result is true also in $\R^d$.
In the earlier works \cite{KL96}, \cite{Kol00a}, \cite{Kol00b}
this result was proved under various extra assumptions.

If $\Lam$ is a lattice, then 
 $\hat\delta_\Lambda = \det(\Lam)^{-1} \cdot \delta_{\Lam^*}$
by the Poisson summation formula. 
This  implies that
$\supp(\hat\delta_\Lambda) = \Lambda^*$.
Hence in this case
the condition \eqref{eqI3.1} is not
only necessary, but also sufficient,
for $f + \Lam$ to be a tiling at some level $w$.

However for a general 
discrete set $\Lambda$ of bounded density,
the sufficiency of the condition \eqref{eqI3.1} for
tiling has remained an open problem.
In this paper, we settle this problem
in the negative. Our main result is the following:

\begin{thm}
\label{thmI8.1}
There is a discrete set $\Lambda \subset \R$  of bounded density (a small perturbation of the integers) with the following property:
given any real scalar $w$ there are two real-valued functions 
$f, g \in L^1(\mathbb{R})$
whose Fourier transforms have the same set of zeros,
but such that $f + \Lambda$ is a tiling at level $w$ 
while $g + \Lambda$ is not a tiling at any level.
\end{thm}

Moreover, we will show that if the given scalar $w$ is 
\emph{positive}, then the functions $f, g$ can be chosen positive 
as well.

It follows that the necessary
condition \eqref{eqI3.1} is generally  not sufficient
for tiling:

\begin{corollary}
\label{corI8.2}
There exist a set $\Lambda \subset \R$  of bounded density
and a positive function $f \in L^1(\mathbb{R})$,
such that \eqref{eqI3.1} is satisfied however $f + \Lam$ is not a tiling
at any level.
\end{corollary}

But even stronger,  \thmref{thmI8.1}
shows that also no other condition can be given
in terms of the Fourier  zero set  $\zfft$ 
that would characterize
 the functions $f \in L^1(\mathbb{R})$ such
that $f + \Lam$ is a tiling,
even under the extra assumption
that $f$ is positive.

Our approach is based on the relation of the problem to
 Malliavin's non-spectral synthesis example \cite{Mal59c}.
The proof  involves the implicit function method
 due to Kargaev \cite{Kar82}, who proved the existence
of a set $\Omega\subset \mathbb R$ of finite measure
 such that the Fourier transform of its indicator function
vanishes on some interval.


\section{Preliminaries. Notation.}

In this section we recall some preliminary background 
and fix notation that will be used later on. 
For further details we refer the reader to \cite{Kah70}.

The closed support of 
 a Schwartz distribution $S$, or a function $\phi$,
on the real line $\R$ or on the circle  $\T = \R / \Z$,
is denoted by $\supp(S)$ or $\supp(\phi)$ respectively.

If $S$ is a Schwartz distribution on $\T$,
its Fourier coefficients $\ft{S}(n)$ are defined by
\[
\ft{S}(n) = \dotprod{S}{e^{-2 \pi i n t}}, \quad n \in \Z.
\]
The action of $S$ on a function $\phi \in C^\infty(\T)$
is denoted by $\dotprod{S}{\phi}$. We have
\begin{equation}
\label{eqR9.1}
\dotprod{S}{\phi} =  \sum_{n \in \Z} \ft{S}(n)\ft{\phi}(-n).
\end{equation}

Let $A(\T)$ be the
Wiener space of continuous functions $\phi$ on   $\T$ 
whose Fourier series converges absolutely.
 It is a Banach space endowed with the norm
\[
\|\phi\|_{A(\T)} = \sum_{n \in \Z} |\ft{\phi}(n)|.
\]

A distribution $S$ on $\T$ is called a \emph{pseudomeasure}
if $S$ can be extended
to a continuous linear functional on $A(\T)$. 
This is the case if and only if
the Fourier coefficients $\ft{S}(n)$ are bounded.
The space  $PM(\T)$ of  all pseudomeasures  
is a Banach space with the norm
\[
\|S\|_{PM(\T)} = \sup_{n \in \Z} |\ft{S}(n)|.
\] 
The duality
between the spaces $A(\T)$ and $PM(\T)$
is given by
\begin{equation}
\label{eqR9.2}
\dotprod{S}{\phi} =  \sum_{n \in \Z} \ft{S}(n)\ft{\phi}(-n),
\quad S \in PM(\T), \; \phi \in A(\T),
\end{equation}
which is consistent with \eqref{eqR9.1}.

In a similar way, we will denote by $A(\R)$ the
space of Fourier transforms of functions in $L^1(\R)$,
that is, $\phi \in A(\R)$ if and only if
\[
\phi(t) = \int_{\R} \ft{\phi}(x) e^{2 \pi i t x} dx, \quad
\ft{\phi} \in L^1(\R),
\quad
\|\phi\|_{A(\R)} = \|\ft{\phi}\|_{L^1(\R)}.
\]
The Banach space dual to $A(\R)$ is then 
the space $PM(\R)$ 
of   temperate distributions $S$ on $\R$ whose
Fourier transform $\ft{S}$ is in $L^\infty(\R)$.
The space $PM(\R)$ is normed as
\[
\|S\|_{PM(\R)} = \|\ft{S}\|_{L^\infty(\R)},
\]
and the duality
between the spaces $A(\R)$ and $PM(\R)$
is given by
\begin{equation}
\label{eqR9.3}
\dotprod{S}{\phi} =  \int_\R \ft{S}(x)\ft{\phi}(-x)dx,
\quad S \in PM(\R), \; \phi \in A(\R).
\end{equation}
The elements of the space $PM(\R)$
are called \emph{pseudomeasures} on $\R$.

The product $\phi \psi$ of two functions $\phi, \psi \in A$
(on either $\T$ or $\R$) is also in $A$, and
\[
\|\phi \psi\|_{A} \leq \|\phi\|_{A} \|\psi\|_{A}.
\]
If $S \in PM$ and $\phi \in A$, then
the product $S\phi$ is a pseudomeasure  defined by
\[
\dotprod{S\phi}{\psi} = \dotprod{S}{\phi \psi},
\quad \psi \in A,
\]
and we have
\[
\|S\phi\|_{PM} \leq \|S\|_{PM} \|\phi\|_{A}.
\]

If $S \in PM$, $\phi \in A$ and if $\phi$ vanishes
in a \emph{neighborhood} of $\supp(S)$, then $S\phi=0$. 
This is obvious from the definition 
of $\supp(S)$ if $\phi$ is a smooth function
of compact support, while for a general $\phi \in A$ this 
 follows by approximation.

If $S$ is a Schwartz distribution on $\R$ supported on 
a compact interval $I=[a,b]$,  then
its Fourier transform $\ft{S}$ is 
 an infinitely smooth function on $\R$ given by
\[
\ft{S}(x) = \dotprod{S}{e^{-2 \pi i x t}}, \quad x \in \R.
\]
(In fact, $\ft{S}$ is the restriction
to $\R$ of an entire function of exponential type).

If $S$ is a distribution on $\R$ supported 
on an interval $I$ of length $|I| < 1$,
then $S$ may  be considered also as a distribution on $\T$,
and in this case we have
$S \in PM(\T)$ if and only if
$S \in PM(\R)$.
If, in addition, $\phi$ is   a  function on $\R$
such that
$\supp(\phi)\sbt I$, then
$\phi \in A(\T)$ if and only if $\phi \in A(\R)$,
and the action $\dotprod{S}{\phi}$ then has the same
value with respect to either definition
\eqref{eqR9.2} or \eqref{eqR9.3}.


\section{Malliavin's non-spectral synthesis phenomenon} \label{secM1}

\subsection{}
The \emph{spectral synthesis problem},  posed by
Beurling, asks the following:
Let $V$ be a closed,
 linear subspace of the space $\ell^\infty(\Z)$
endowed with the weak* topology (as the dual of $\ell^1$).
We say that $V$ is \emph{translation-invariant}
if whenever a sequence $\{c(n)\}$ 
belongs to $V$, then so do all of the translates of
 $\{c(n)\}$.
Define the  \emph{spectrum} $\sigma(V)$ of a
translation-invariant subspace $V$ to be
the (closed) set of points $t \in \T$
such that the  sequence
$e_t := \{\exp (2 \pi i n t)\}$ is in $V$.
Is it true that $V$ is generated by the exponentials
$e_t$, $t \in \sig(V)$, i.e.\  is $V$ the weak* closure
of the linear span of these exponentials?

There are 
also other, equivalent
formulations of the 
spectral synthesis problem,
see  \cite[Chapter IX]{KS94}.
 One of them is the following:
Let $S \in PM(\T)$, $\phi \in A(\T)$, and
assume that $\phi$ vanishes on $\supp(S)$.
Does it follows that
$\dotprod{S}{\phi}=0$?

The answer to the last question is affirmative if $\phi$ is smooth, or,
more generally, if $\phi \in A(\T) \cap \operatorname{Lip}(\half)$.
This result is due to Beurling and Pollard, 
see e.g.\ \cite[Chapter V, Section 5]{Kah70}.
However, it was proved by Malliavin
  that in the general case,
the question admits a negative answer:

 \begin{thm}[Malliavin \cite{Mal59a}, \cite{Mal59b}]
\label{thmM1.1}
There exist 
a pseudomeasure $S \in PM(\T)$
and 
a function $\phi \in A(\T)$ 
such that
$\phi$ vanishes on $\supp(S)$,
but $\dotprod{S}{\phi} \neq 0$.
\end{thm}

The spectral synthesis problem can be posed more generally 
in any locally compact abelian group $G$
(where the case discussed above corresponds to the
group $G=\Z$).
For compact groups the problem admits a positive answer;
while Malliavin showed \cite{Mal59c} that the answer
is negative for all non-compact groups $G$.

For more details on the subject we refer the reader
to \cite[Chapter IX]{KS94},
\cite[Chapter V]{Kah70},
\cite[Chapter 7]{Rud62},
\cite{Ben75},
\cite[Chapter 3]{GM79}.

\subsection{}
Let $S \in PM(\T)$ and
$\phi \in A(\T)$ 
be given by Malliavin's theorem (\thmref{thmM1.1}), that is,
$\phi$ vanishes on $\supp(S)$
while $\dotprod{S}{\phi} \neq 0$.
Since $\phi$ does not vanish everywhere on the circle $\T$, 
there is  an open interval $I$ of length $|I| <1$ such that
$\supp(S) \sbt I$. Hence we may
regard $S$ also as a distribution on $\R$, and
we have $S \in PM(\R)$.
By multiplying $\phi$ on a smooth function supported
on $I$ and which
is equal to $1$ in a neighborhood of $\supp(S)$,
we may assume that $\supp(\phi) \sbt I$ as well,
and consequently $\phi \in A(\R)$.

Furthermore, by applying a linear change of variable
to $S$ and $\phi$, we may actually suppose that $I$ is 
an arbitrary open interval on $\R$. We shall 
take $I = (a,b)$
where $a, b$ are any two numbers satisfying $0<a<b<\half$.

For each $r>0$ we now define a distribution 
$T_r \in PM(\R)$ by
 \begin{equation}
\label{eqR14.1}
T_r  := \delta_0 + r (S + \widetilde{S}),
 \end{equation}
where $\widetilde{S}(t) := \overline{S(-t)}$.\footnote{The
distribution $\widetilde{S}$ can be more formally defined by
$\dotprod{\widetilde{S}}{\psi} := 
\overline{\dotprod{S}{\widetilde{\psi}}}
$ where $\widetilde{\psi}(t) := \overline{\psi(-t)}$.}
We will prove the following result:

 \begin{thm}
\label{thmR7.15}
Given any  $\eps >0$ there exists
 a real sequence $\Lam = \{\lam_n\}$, $n\in\Z$,
satisfying $|\lam_n - n | \leq \eps$ for all $n$,
such that for some $r>0$
we have $\ft{\delta}_\Lam = T_r$ in the interval $(-b,b)$.
 \end{thm}

The proof of this theorem will be given in the next section. 
 Our goal in the present section is to complete the
proof of \thmref{thmI8.1} based on this result.
We will show that $\Lam$ 
has the property from the statement of the theorem:
given any real scalar $w$ there are two real-valued functions 
$f, g \in L^1(\mathbb{R})$
whose Fourier transforms have the same set of zeros,
but such that $f + \Lambda$ is a tiling at level $w$ 
while $g + \Lambda$ is not a tiling at any level.
Moreover, if the given scalar $w$ is 
\emph{positive}, then the
functions $f, g$ can be chosen positive
as well.

\subsection{}
Since the set $\Lam$ has bounded density,
 for any  $h \in L^1(\R)$ the convolution
$h \ast \delta_\Lam$ is a locally integrable
function satisfying
\begin{equation}
  \label{eqP1.13}
\sup_{x \in \R} \int_{x}^{x+1} |(h \ast \delta_\Lam)(y)| dy < +\infty,
\end{equation}
see \cite[Lemma 2.2]{KL96}. This implies that $h \ast \delta_\Lambda$ 
is a temperate distribution on $\mathbb R$.

\begin{lem}
  \label{lemR14.3}
Let $h$ be a function in $L^1(\R)$ such that $\supp(\ft{h}) \sbt (-b,b)$.
Then the Fourier transform of $h \ast \delta_\Lam$
is the pseudomeasure $T_r  \cdot \ft{h}$.
\end{lem}

\begin{proof}
The assertion means that for any
Schwartz function $\beta$ we have
\begin{equation}
  \label{eqP1.10}
\int_{\R} (h \ast \delta_\Lam)(x) \, \ft{\beta}(x) \, dx = 
\dotprod{T_r}{\ft{h} \cdot \beta}.
\end{equation}

Let $\chi$ be a Schwartz function whose Fourier transform
$\ft{\chi}$ is nonnegative, has compact support,
 $\int \ft{\chi}(t) dt =1$, and for each $\eps > 0$ let
$\chi_\eps(x) := \chi( \eps x)$. Let
$q_\eps := (\ft{h} \cdot \beta) \ast \ft{\chi}_\eps$,
then $q_\eps$ is an infinitely smooth function with compact support.
As $\eps \to 0$, the function $q_\eps$ remains supported on
a certain closed interval $J$ contained in $(-b,b)$, and
$q_\eps$ converges to 
$\ft{h}\cdot   \beta$ in the space $A(\R)$. The assumption that
 $\ft{\delta}_\Lam = T_r$ in $(-b,b)$ thus implies that
\begin{equation}
  \label{eqR12.3}
\lim_{\eps \to 0} \dotprod{\ft{\delta}_\Lam}{q_\eps} = 
\lim_{\eps \to 0} \dotprod{T_r}{q_\eps} = 
\dotprod{T_r}{\ft{h}\cdot  \beta}.
\end{equation}

The function $\beta$ is the Fourier transform of some function
$\alpha$ in the Schwartz class. Let $p_\eps := (h \ast \alpha) \cdot \chi_\eps$,
then $p_\eps$ is a smooth function in $L^1(\R)$ and
we have $\ft{p}_\eps = q_\eps$.
Since $q_\eps$ belongs to the Schwartz 
space, the same is true for $p_\eps$, and it follows that

\begin{align}
\dotprod{\ft{\delta}_\Lam}{q_\eps} &= \dotprod{\delta_\Lam}{\ft{q}_\eps}
= \sum_{\lambda\in\Lambda}  {p_\eps}(-\lam)
= \sum_{\lambda\in\Lambda}
   (h \ast \alpha)(-\lam) \, \chi_\eps(-\lam)\nonumber\\
&=\sum_{\lambda\in\Lambda} \chi_\eps(-\lam) \int_{\mathbb R}\alpha(-x) h(x-\lambda)dx.
  \label{eqR12.7}
\end{align}

Now we need the following:
\begin{claim*}
We have
\begin{equation}\label{eqR14.21}
\sum_{\lambda\in\Lambda} 
 \int_{\mathbb R}|\alpha(-x)|\cdot|h(x-\lambda)|dx<+\infty.
\end{equation}
\end{claim*}

We observe that $|\chi_\eps(-\lam)| \leq 1$ and
$\chi_\eps(-\lam) \to 1$ as $\eps \to 0$ for each $\lam$.
Hence the claim allows us to apply 
the  dominated convergence theorem
to the sum \eqref{eqR12.7}, which yields
\begin{equation}
  \label{eqR12.11}
\lim_{\eps \to 0} \dotprod{\ft{\delta}_\Lam}{q_\eps} 
= \sum_{\lambda\in\Lambda}  \int_{\mathbb R}\alpha(-x)  h(x-\lambda)dx.
\end{equation}
The claim also allows us to exchange the sum and integral in \eqref{eqR12.11},
and it follows that
\begin{equation}
  \label{eqR12.12}
\lim_{\eps \to 0} \dotprod{\ft{\delta}_\Lam}{q_\eps} 
= \int_{\mathbb R}\alpha(-x) \sum_{\lambda\in\Lambda} h(x-\lambda)dx
= \int_{\R} (h \ast \delta_\Lam)(x) \, \ft{\beta}(x) \, dx.
\end{equation}
Comparing \eqref{eqR12.3} and
\eqref{eqR12.12}, we see that
\eqref{eqP1.10} holds.

It remains to prove the claim. Indeed,  we have
\begin{equation}\label{eqR14.20}
\sum_{\lambda\in\Lambda} 
 \int_{\mathbb R}|\alpha(-x)|\cdot|h(x-\lambda)|dx
= \int_{\mathbb R}|h(-x)|
\sum_{\lambda\in\Lambda} 
|\alpha(x-\lambda)|dx.
\end{equation}
The inner sum on the right hand side of \eqref{eqR14.20}
 is a bounded function of $x$, since $\alpha$ is a 
Schwartz function and $\Lambda$ has bounded density, 
while $h$ is a function in $L^1(\mathbb R)$.
 Hence the integral in 
\eqref{eqR14.20} converges, and this completes the proof
of the lemma.
\end{proof}

\subsection{}
Recall that $S \in PM(\R)$,
$\supp(S) \sbt (a,b)$  where $0<a<b<\half$,
$\phi \in A(\R)$ is a function
with $\supp(\phi) \sbt (a,b)$,
$\phi$ vanishes on $\supp(S)$, and 
 $\dotprod{S}{\phi} \neq 0$.
Let $\psi$ be a smooth function
whose zero set   $\zeros(\psi)$ is the same
as $\zeros(\phi)$. In particular, we have
$\supp(\psi) \sbt (a,b)$ and
$\psi$ vanishes on $\supp(S)$ as well.
Let also $\tau$ be a smooth
function satisfying
$ \tau(-t) = \overline{ \tau(t)}$,
$\supp(\tau) \sbt (-a,a)$,
and $\tau(0)=1$.

Given a real scalar $w$ we define two functions 
$f, g \in L^1(\mathbb{R})$ by the conditions
 \begin{equation}
\label{eqR14.2}
\ft{f}(t) = w \cdot \tau(t) + \psi(t) + \overline{\psi(-t)},
 \end{equation}
 \begin{equation}
\label{eqR14.3}
\ft{g}(t) = w \cdot \tau(t) + \phi(t) + \overline{\phi(-t)},
 \end{equation}
then $f,g$ are real-valued and their
Fourier transforms have the same set of zeros.

By \lemref{lemR14.3} the Fourier transform of $f \ast \delta_\Lam$ is the pseudomeasure 
\[
\ft{f} \cdot T_r = w \delta_0 + r ( S \psi  + (\widetilde{S \psi)}) =
w \delta_0,
\]
where the first equality is 
due to \eqref{eqR14.1}
and \eqref{eqR14.2},
while the second equality 
is true since $\psi$ is
smooth and vanishes on $\supp(S)$, 
hence $S \psi =0$.
We conclude that 
$f \ast \delta_\Lam = w$ a.e., which means
that  $f + \Lambda$ is a tiling at level $w$.

In the same way,
 \lemref{lemR14.3}  implies 
that  the Fourier transform of $g \ast \delta_\Lam$ is 
\[
\ft{g} \cdot T_r = w \delta_0 + r (S \phi  + (\widetilde{S \phi)}).
\]
However in this case, $S \phi$ is not the
zero distribution, since $\dotprod{S\phi}{1} = 
\dotprod{S}{\phi} \neq 0$.
This shows  that 
the Fourier transform of $g \ast \delta_\Lam$ is 
not a scalar multiple
of $\delta_0$, and it follows that
 $g + \Lambda$ is not a tiling at any level.

\subsection{}
The above construction yields real-valued functions
$f$ and $g$, but 
these two functions need not be positive. We will now show that
if the given scalar $w$ is positive,  then 
the construction  can be modified so as to
yield \emph{everywhere positive} functions $f,g$.

In what follows, $\phi$ and
$\psi$  continue to denote the same two functions as above.

\emph{Step 1}: We show that there is a nonnegative sequence
 $\{c(k)\} \in \ell^1(\Z)$, such that
 \begin{equation}
\label{eqR16.1}
|\ft{\phi}(x)| \leq c(k), \quad k \in \Z, \; |x - k| \leq \half.
 \end{equation}

Indeed, we have $\phi  \in A(\R)$ and
$\supp(\phi) \sbt (a,b)$. Considered as a
function in $A(\T)$, $\phi$ may be expressed
on $(a,b)$ as  the sum
of an absolutely convergent Fourier
series. Hence 
there is a finite (complex) measure $\mu$ supported on $\Z$ 
such that $\phi(t) = \ft{\mu}(-t)$, $t \in (a,b)$.
 Let $\Phi$ be 
an infinitely smooth
 function  such that
$\Phi(t)=1$ if $t \in \supp(\phi)$,
while $\Phi(t)=0$ for $t \in \R \setminus (a,b)$.
Then $\phi(t) = \ft{\mu}(-t) \Phi(t)$ for
every $t \in \R$, which implies that
\[
\ft{\phi}(x) = (\mu \ast \ft{\Phi})(x) = \sum_{n \in \Z} {\mu}(n) \ft{\Phi}(x-n),
\quad x \in \R.
\]
Since the Fourier transform $\ft{\Phi}$ has fast decay, there is
a sequence 
 $\{\gamma(k)\} \in \ell^1(\Z)$ such that
$|\ft{\Phi}(x)| \leq |\gamma(k)|$ whenever
$|x - k| \leq \half$. It follows that
\[
|\ft{\phi}(x)| \leq c(k) := \sum_{n \in \Z} |{\mu}(n)| \cdot | \gamma(k-n)|,
\quad |x-k| \leq \half,
\]
which establishes \eqref{eqR16.1}.

\emph{Step 2}: We may assume that the same sequence
 $\{c(k)\}$ also satisfies
 \begin{equation}
\label{eqR16.6}
|\ft{\psi}(x)| \leq c(k), \quad k \in \Z, \; |x - k| \leq \half.
 \end{equation}

Indeed, we may apply the same procedure from Step 1 
also to the function $\psi$, and then define $\{c(k)\}$
to be the maximum of the
two sequences obtained from both steps.

\emph{Step 3}: Let  $\{d(k)\}$ be any positive sequence
in $\ell^1(\Z)$ such that
 \begin{equation}
\label{eqR16.12}
 d(k) >  c(k), \quad k \in \Z.
 \end{equation}
We show that there is 
$\tau \in A(\R)$ such that
$\supp(\tau) \sbt (-a,a)$ and
 \begin{equation}
\label{eqR16.7}
\ft{\tau}(x) \geq d(k), \quad k \in \Z, \; |x - k| \leq \half.
 \end{equation}

Let $\chi$ be an infinitely smooth
function, $\supp(\chi) \sbt
(-a,a)$, whose
Fourier transform $\ft{\chi}$ is  nonnegative 
and satisfies $\ft{\chi}(x)\geq 1$ on the interval
$[-\half, \half]$. 
Let $\tau(t) := \ft{\nu}(-t) \chi(t)$,
where $\nu$ is  a positive, finite measure supported on $\Z$ 
defined by 
$\nu := \sum_n d(n)\delta_n$. 
Then the function $\ft{\tau} = \nu \ast \ft{\chi}$ is in $L^1(\R)$,
so that we have $\tau \in A(\R)$, and
\[
\ft{\tau}(x)  = \sum_{n \in \Z} d(n) \ft{\chi}(x-n)
\geq  d(k) \ft{\chi}(x-k) \geq d(k),
\quad |x-k| \leq \half,
\]
which gives \eqref{eqR16.7}.

\emph{Step 4}: Now suppose that we are given a \emph{positive} scalar $w$.
We  then define the two functions $f,g$ by the conditions
 \begin{equation}
\label{eqR16.17}
\ft{f}(t) = w \cdot \tau(0)^{-1} \cdot \Big[ \tau(t) +  \frac{\psi(t) + \overline{\psi(-t)}}{2} \Big],
 \end{equation}
 \begin{equation}
\label{eqR16.18}
\ft{g}(t) = w \cdot \tau(0)^{-1} \cdot \Big[ \tau(t) +  \frac{\phi(t) + \overline{\phi(-t)}}{2} \Big].
 \end{equation}
We observe that by the definition of the function $\tau$ we have
\[
\textstyle
\tau(0) = \ft{\nu}(0) \chi(0) = (\int d\nu)(\int \ft{\chi}(x) dx) > 0,
\]
and in particular $\tau(0)$ is nonzero. 
Then $f,g$ are in $L^1(\R)$, their
Fourier transforms have the same set of zeros,
and by the same argument as before one can
verify that  $f + \Lambda$ is a tiling at level $w$,
while $g + \Lambda$ is not a tiling at any level.

Finally we check that
$f$ and $g$ are everywhere positive functions. Indeed, we have
\[
f(x) =
w \cdot \tau(0)^{-1} \cdot \Big[ \ft{\tau}(-x) + 
 \Re(\ft{\psi}(-x) ) \Big],
\]
\[
g(x) =
w \cdot \tau(0)^{-1} \cdot \Big[ \ft{\tau}(-x) + 
 \Re(\ft{\phi}(-x) ) \Big],
\]
and by \eqref{eqR16.1}, \eqref{eqR16.6},
\eqref{eqR16.12} and \eqref{eqR16.7}
 it follows that 
$f(x), g(x) >0$ for every $x \in \R$.

This completes the proof of \thmref{thmI8.1}
based on \thmref{thmR7.15}.
\qed

It remains to prove \thmref{thmR7.15}.
This will be done in the next section.


\section{Kargaev's implicit function method} \label{secR2}

\subsection{}
In \cite{Sap78}, Sapogov posed the 
following question: Does there exist a set $\Omega \sbt \R$
of positive and  finite  measure, 
such that the Fourier transform of its indicator function 
$\1_\Om$  vanishes on some open interval $(a,b)$?

 The question
was answered in the affirmative by 
Kargaev \cite{Kar82}. The solution was based on
an innovative application
 of the infinite-dimensional implicit function 
theorem, which established the existence of a 
set of the form
$ \Om = \bigcup_{n\in\Z} [n+\alpha_n, n+\beta_n]$,
where $\{\alpha_n\}$, $\{\beta_n\}$
are two real sequences in $\ell^1$, 
that has the above mentioned property.

The approach was later used in \cite{KL16} in
order to prove the existence of 
non-periodic tilings of $\R$ by 
 translates of a function $f$.
In that paper,  a self-contained
presentation of the method was given
in a simplified form, that does not invoke
 the infinite-dimensional implicit function
theorem.

In this section, we use an adapted version
of Kargaev's method in order to prove 
 \thmref{thmR7.15} (and more, in fact).
The presentation below generally 
follows the lines of \cite[Sections 2, 3]{KL16},
but the proof also requires some additional arguments.

\subsection{}
 Let $\{\alpha_n\}$, $n\in\Z$, be a bounded sequence of real numbers. To such a sequence
 we associate a function $F$ on the real line, defined by
 \begin{equation}
\label{eqR7.1}
F(x)=\sum_{n\in\mathbb Z} F_n(x), \quad x\in\R,
 \end{equation}
 where $F_n$ is the function $\1_{[n,n+\alpha_n]}$ if $\alpha_n \geq 0$,
or  $-\1_{[n+\alpha_n,n]}$ if $\alpha_n<0$.

Since the sequence $\{\alpha_n\}$ is bounded, the series
\eqref{eqR7.1} is easily seen to converge in the space
of temperate distributions to a bounded function $F$ on $\R$.
In particular, $F$ is a temperate distribution.

 \begin{thm}
\label{thmR7.2}
Given two numbers $b \in (0,\half)$ and $\eps >0$,
there is $\delta>0$
with the following property:
Let $S$ be a Schwartz distribution on $\R$ satisfying
 \begin{equation}
\label{eqR7.2.11}
S(-t) = \overline{S(t)}, \quad
\supp(S) \sbt (-b,b), \quad
\sup_{k \in \Z} |\ft{S}(k)| \leq \delta.
\end{equation}
Then there is a bounded, real sequence $\alpha = \{\alpha_n\}$, $n\in\Z$,
such that $\|\alpha\|_\infty \leq \eps$ and
 \begin{equation}
\text{$\hat F=S$ in $(-b,b)$,} 
\end{equation}
where $F$ is the function  defined by \eqref{eqR7.1}.
 \end{thm}

The proof of \thmref{thmR7.2} is given
below. It is divided into a series of lemmas.

\subsection{}
Given a number $b \in (0,\half)$ we choose $l = l(b)$ such that
 $ b<l< \half$. We also choose an infinitely smooth
 function $\Phi$ satisfying
the conditions
$\Phi(-t)=\overline{\Phi(t)}$ for all $t \in \R$,
$\Phi(t)=1$ for $t\in[-b,b]$, and
$\Phi(t)=0$ for $t \in \R \setminus (-l,l)$.

Let $I := [-\half, \half]$. Denote by
$\ft{\psi}(k)$ the $k$'th Fourier
coefficient of a function $\psi$ on $I$:
\begin{equation}
\label{eqR1.1.10}
\ft{\psi}(k) = \int_{I} \psi(t) e^{-2\pi i k t} dt, \quad k \in \Z.
\end{equation}

The following lemma is inspired by \cite[Lemma 2.2]{KV92}.

\begin{lem}
  \label{lemR1.1}
Let $\varphi \in C^\infty(\R)$,
 $\varphi (0)=0$, and denote 
$\psi_s(t) := \varphi(st) \Phi(t)$. Then
\begin{equation}
\label{eqR1.1.6}
| \ft{\psi}_s(k) | \leq \frac{C |s|}{1+|k|^m}
\end{equation}
for every $s \in [-1,1]$ and $k \in \Z$,
where $C = C(\Phi, \varphi, m)>0$ is a constant 
 which depends neither on $s$ nor on $k$.
\end{lem}

\begin{proof}
First suppose that $k=0$.  We have
$|\varphi(t)| \leq C|t|$ for $t \in I$, hence
\begin{equation}
\label{eqR1.1.1}
|\ft{\psi}_s(0)| \leq C
|s| \int_{I} |t \Phi(t)| dt = C |s|.
\end{equation}

Next we assume that $k \neq 0$. We
integrate by parts $m$ times and use the
fact that the function $\psi_s$ vanishes in
a neighborhood of the points $\pm \half$.
This yields
\begin{equation}
\label{eqR1.1.2}
\ft{\psi}_s(k) = \frac1{(2 \pi i k)^m}
\int_{I}  \psi_s^{(m)}(t)
 e^{-2\pi i k t} dt.
\end{equation}
By the product rule for the $m$'th derivative we have
\begin{equation}
\label{eqR1.1.3}
  \psi_s^{(m)}(t)
= \sum_{j=0}^{m} \binom{m}{j}
s^j \varphi^{(j)} (st) \Phi^{(m-j)}(t).
\end{equation}
Combining \eqref{eqR1.1.2} and \eqref{eqR1.1.3} yields
the estimate
\[
| \ft{\psi}_s(k)| \leq \frac1{(2 \pi |k|)^m}
 \sum_{j=0}^{m} \binom{m}{j}
|s|^j \int_I |\varphi^{(j)} (st) \Phi^{(m-j)}(t)| dt.
\]
Since the derivatives
$\varphi', \varphi'', \dots, \varphi^{(m)}$
are bounded on $I$, each one of the terms in the sum
corresponding to $j=1,2,\dots,m$ is bounded by
$C |s|$, while the term
corresponding to $j=0$ can be estimated using
$|\varphi(t)| \leq C|t|$, $t \in I$, which again
yields $C |s|$.
\end{proof}

\begin{lem}
  \label{lemR1.5}
Let $\varphi \in C^\infty(\R)$,
 $\varphi'(0)=0$, and denote 
\begin{equation}
\psi_{u,v}(t) := \frac{\varphi(vt) - \varphi(ut)}{t} \cdot \Phi(t).
\end{equation}
 Then
\begin{equation}
\label{eqR1.5.6}
| \ft{\psi}_{u,v}(k) | \leq  \max \{|u|,|v|\} \cdot \frac{C |v-u| }{1+|k|^m}
\end{equation}
for every $u,v \in [-1,1]$ and
$k \in \Z$,
where $C = C(\Phi, \varphi, m)>0$ is a constant 
 which does not depend on $u$, $v$ or $k$.
\end{lem}

\begin{proof}
We may suppose that $u<v$. We observe that
\begin{equation}
\label{eqR1.5.1}
\psi_{u,v}(t) = \int_u^v \varphi'(st) \Phi(t) ds
= \int_u^v \psi_s(t) ds,
\end{equation}
where we define $\psi_s(t) := \varphi'(st) \Phi(t)$. Hence
\begin{equation}
\label{eqR1.5.2}
\ft{\psi}_{u,v}(k) =  \int_u^v \ft{\psi}_s(k) ds,
\quad k \in \Z.
\end{equation}
By \lemref{lemR1.1} the estimate
\eqref{eqR1.1.6} is valid for every $s \in [u,v]$,
where $C = C(\Phi, \varphi, m)>0$ is a constant 
 which does not depend on $s$ or $k$. Hence
\begin{equation}
\label{eqR1.5.3}
|\ft{\psi}_{u,v}(k)| \leq  \frac{C}{1+|k|^m}
\int_u^v |s| ds
\end{equation}
from which \eqref{eqR1.5.6} follows.
\end{proof}

\subsection{}
If $T$ is a Schwartz distribution 
supported on $[-l,l]$, then $\ft{T}(k)$ denotes the
 $k$'th Fourier coefficients of $T$:
\[
\ft{T}(k) = \dotprod{T}{ e^{-2\pi i k t}},
\quad k \in \Z.
\]

\begin{lem}
  \label{lemR2.13}
Let $T$ be a distribution 
supported on $[-l,l]$. Then the series
\begin{equation}
\label{eqR2.13.2}
\sum_{n \in \Z} \ft{T}(n) e^{2 \pi i n t}
\end{equation}
converges unconditionally in the distributional sense
to $T$ in the open interval $(-\half, \half)$.
\end{lem}

This follows from the unconditional convergence of the series
\eqref{eqR2.13.2} to $T$ considered
as a distribution on the circle $\T$.

\subsection{}
Let $X$ be the space of all bounded sequences of real numbers
$\alpha = \{\alpha_n\}$,  $n \in \Z$,
endowed with the norm
\[
\|\alpha\|_X := \sup_{n \in \Z} |\alpha_n|
\]
that makes $X$ into a real Banach space.

Let $Y$ be the space of distributions $T$
supported on $[-l,l]$ whose  Fourier coefficients
$\ft{T}(k)$, $k \in \Z$,
are real and bounded. If we endow $Y$ with the norm
\begin{equation}
\label{eqR2.7.12}
\|T\|_Y := \sup_{k \in \Z} |\ft{T}(k)|
\end{equation}
then also $Y$ is a real Banach space,
which can be viewed as a closed subspace
of $PM(\T)$.

We observe that a distribution $T$
supported on $[-l,l]$ has real Fourier coefficients
(that is, $\ft{T}(k) \in \R$ for every
 $k \in \Z$) if and only if $T(-t) = \overline{T(t)}$.

\begin{lem}
  \label{lemR2.7}
Let $\{T_n\}$, $n \in \Z$, be a sequence of elements
of $Y$. Assume that there is 
a sequence $\gam \in X$ such that
\begin{equation}
\label{eqR2.7.1}
|\ft{T}_n(k)| \leq \frac{|\gam_n|}{1+|k-n|^2}
\end{equation}
for every $n$ and $k$ in $\Z$.
Then the series  
\begin{equation}
\label{eqR2.7.2}
T=\sum_{n \in \Z} T_n
\end{equation}
converges unconditionally in the distributional sense
to an element $T \in Y$
satisfying 
\begin{equation}
\label{eqR2.7.21}
\|T\|_Y \leq K \|\gam\|_X
\end{equation}
where $K$ is an absolute constant.
\end{lem}

\begin{proof}
Indeed, the condition
\eqref{eqR2.7.1} implies that
for any $\phi \in A(\T)$  we have
\[
 |\dotprod{T_n}{\phi}|
=  \Big| \sum_{k \in \Z} \ft{T}_n(k) \ft{\phi}(-k) \Big|
\leq |\gam_n| \sum_{k \in \Z} \frac{|\ft{\phi}(-k)|}{1+|k-n|^2},
\]
and hence
\[
\sum_{n \in \Z} |\dotprod{T_n}{\phi}|
\leq K \|\gamma\|_X \|\phi\|_{A(\T)},
\quad
K := \sum_{n \in \Z} \frac{1}{1+|n|^2}.
\]
This shows that
the series   \eqref{eqR2.7.2}
converges unconditionally in the weak* topology
of the space $PM(\T)$ (the dual of $A(\T))$
to an element $T \in Y$ satisfying 
\eqref{eqR2.7.21}.
\end{proof}

\subsection{}
Let $\alpha = \{\alpha_n\}$, $n \in \Z$, be a  sequence  in $X$
such that $\|\alpha\|_X \leq 1$. Define 
\begin{equation}
\label{eqR2.3}
(R \alpha)(t):=\sum_{n\in\mathbb Z} e^{2\pi int}\cdot \frac{e^{2\pi i\alpha_n t}-1-2\pi i\alpha_n t}{2\pi it} \cdot \Phi(t).
\end{equation}
Let $T_n$ be the $n$'th term of the series
\eqref{eqR2.3}. We observe that $T_n \in Y$.
 If we apply \lemref{lemR1.1} to the function
$\varphi(t) :=  (e^{2\pi i t}-1 -  2 \pi i t)/(2\pi it)$
with $s = \alpha_n$ and $m=2$, then
it follows from the lemma that
 condition \eqref{eqR2.7.1} is satisfied
with $\gam_n := C \alpha_n^2$,
 where $C>0$ does not
 depend on $\alpha$, $k$ or $n$.
Hence by \lemref{lemR2.7} 
the series \eqref{eqR2.3} converges
in the distributional sense to an
element of  the space $Y$,
and we have
\begin{equation}
\label{eqR2.4.1}
\|R \alpha\|_Y \leq C \|\alpha\|^2_X,
\quad \alpha \in X, \; \|\alpha\|_X \leq 1,
\end{equation}
where the constant $C$ does not depend on $\alpha$.

We note that the mapping $R$ 
defined by \eqref{eqR2.3} is  \emph{nonlinear}.

\subsection{}
For each $r>0$ let
 $U_r$ denote the closed ball of radius $r$ around the origin
in $X$:
\begin{equation}
\label{eqR3.1.1}
U_r := \{\alpha \in X : \|\alpha\|_X \leq r\}.
\end{equation}

\begin{lem}
\label{lemR3.1}
Given any $\rho > 0$ there is $0<r<1$ such that
\begin{equation}
\label{eqR3.1.2}
\|R\beta - R\alpha\|_Y \leq \rho \|\beta-\alpha\|_X,
\quad \alpha,\beta \in U_r.
\end{equation}
In particular, if $r$ is small enough then $R$ is
a contractive (nonlinear) mapping on $U_r$.
\end{lem}

\begin{proof}
Let $\alpha,\beta \in U_r$ $(0<r < 1)$. Then using
\eqref{eqR2.3}  we  have
\begin{equation}
\label{eqR3.2}
(R \beta - R \alpha)(t) =
\sum_{n\in\mathbb Z} e^{2\pi int}\cdot \frac{(e^{2\pi i \beta_n t}
-2\pi i \beta_n t)
- ( e^{2\pi i\alpha_n t}
-2\pi i \alpha_n t) }{2\pi it} \cdot \Phi(t).
\end{equation}
Let $T_n$ be the $n$'th element of the series
\eqref{eqR3.2}. 
 We apply \lemref{lemR1.5} to the function
$\varphi(t) :=  (e^{2\pi i t} -  2 \pi i t)/(2\pi i)$
 with $u = \alpha_n$, 
$v = \beta_n$ and $m=2$.
The lemma implies that
the condition \eqref{eqR2.7.1} is satisfied
with $\gam_n := C r \cdot (\beta_n - \alpha_n)$,
 where the constant $C$ does not
 depend on $r$, $\alpha$, $\beta$, $k$ or $n$.
It therefore follows from \lemref{lemR2.7} that we have
the estimate
$\|R\beta - R\alpha\|_Y \leq Cr  \|\beta-\alpha\|_X$
where $C$ is a constant not depending on $r$, $\alpha$ or $\beta$.
Hence it suffices to choose $r$ small enough so that  $C r \leq  \rho$.
\end{proof}

\subsection{}
For each element $T \in Y$ we denote by $\F(T)$ 
the sequence of Fourier coefficients  of $T$, 
namely, the  sequence 
$\{\ft{T}(k)\}$, $k \in \Z$.
This defines a linear mapping  $\F: Y \to X$ satisfying
$\|\F(T)\|_X = \|T\|_Y$.

\begin{lem}
\label{lemR4.2}
Given any $\eps>0$ there is 
$\delta>0$ with the following property: 
Let $S \in Y$, 
$\|S\|_Y \le  \delta$. Then one can find an element
$T \in Y$,  $\|T-S\|_Y \le \eps \|S\|_Y$,  which solves the
equation  $T + R(\F(T)) = S$.
\end{lem}

 \begin{proof} 
Fix $S \in Y$ such that  $\|S\|_Y \le \delta$,
and let 
\[
B = B(S,\eps) :=\{ T \in Y : \|T-S\|_Y \le \eps \|S\|_Y\}.
\]
We observe that if $T \in B$ then  $\|T\|_Y \le (1+\eps) \|S\|_Y$.
Define a map $H: B \to Y$ by
\[
H(T) := S - R(\F(T)), \quad T \in B,
\]
and notice that an element $T \in B$ is a solution to the equation
$T + R(\F(T)) = S$ 
if and only if $T$ is a fixed point of the map $H$.
 
Let us show that if $\delta$ is small enough then $H(B)\subset B$.
 Indeed, if $T \in B$ then using \eqref{eqR2.4.1} we have
\[
\|H(T)-S\|_Y = \|R(\F(T))\|_Y 
\leq C \|\F(T)\|^2_X =
 C \|T\|^2_Y \leq C (1+\eps)^2 \|S\|^2_Y.
\]
Hence if we choose $\delta$ such that 
$C(1+\eps)^2 \delta \leq \eps$ then
we obtain
\[
\|H(T)-S\|_Y \leq \eps \|S\|_Y,
\]
and it follows that $H(B)\subset B$.
 
It also follows  from \lemref{lemR3.1} 
that if $\delta$ is small enough, then $H$ is a contractive 
mapping from the closed set $B$ into itself.
Indeed, let $T_1, T_2 \in B$, then we have
\[
\|H(T_2) - H(T_1) \|_Y = 
\|R(\F(T_2)) - R(\F(T_1)) \|_Y \le
\rho \|\F(T_2) - \F(T_1) \|_X
= \rho \|T_2 - T_1 \|_Y,
\]
where $0<\rho<1$.
Then the Banach fixed point theorem implies that
$H$ has a (unique) fixed point $T \in B$,
which yields the desired solution.
 \end{proof}

\subsection{}
\emph{Proof of \thmref{thmR7.2}}.
Let $S$ be a Schwartz distribution satisfying
\eqref{eqR7.2.11}. Then $S \in Y$
and $\|S\|_Y \le  \delta$. Define
$S_1(t) := S(-t)$, then also $S_1$ is
a distribution in $Y$ and we have $\|S_1\|_Y = \|S\|_Y$.
By \lemref{lemR4.2}, if $\delta$ is small 
enough then there is an element
$T \in Y$,  $\|T-S_1\|_Y \le \eps \|S_1\|_Y$,  which solves the
equation  $T + R(\F(T)) = S_1$.

Let $\alpha \in X$ be the sequence defined by
 $\alpha_n=\ft{T}(n)$, $n\in\Z$, then
$\|\alpha\|_X \leq \eps$ provided that
 $\delta$ is small enough. Let
 $F$ be the function given by \eqref{eqR7.1} that is
associated to this sequence $\alpha=\{\alpha_n\}$. We have
\[
\hat F (-t) = \lim_{N\to\infty}\sum_{|n|\le N}\hat F_n(-t)
\]
 in the sense of distributions, and
\[
\hat F_n(-t) = e^{2\pi int} \cdot \frac{e^{2\pi i\alpha_n t}-1}{2\pi it}.
\]
Hence
\begin{align*}
\hat F(-t) =
\lim_{N\to\infty} \Big[ \sum_{|n|\le N} \alpha_n e^{2\pi int} 
 + \sum_{|n|\le N} e^{2\pi int}  \cdot
\frac{e^{2\pi i\alpha_n t}-1-2\pi i\alpha_n t}{2\pi it} \, \Big].
\end{align*}
The first sum converges to $T$ in $(-b,b)$ according to
  \lemref{lemR2.13}; while the second sum  converges 
to $R\alpha$ in $(-b,b)$, which is due to \eqref{eqR2.3} and
the fact that
$\Phi(t)=1$ on $(-b,b)$. We conclude that
\[
 \hat F(-t)= (T + R\alpha)(t) = S_1(t)
\quad \text{in $(-b,b)$.}
\]
This means that $\ft{F} = S$ in $(-b,b)$ and thus
\thmref{thmR7.2} is proved.
\qed

\subsection{}
The theorem just proved will now
be used to deduce the following one:

 \begin{thm}
\label{thmR7.5}
Given two numbers $a,b$ such that $0 <a<b < \half$,
and given $\eps>0$,  there is $\delta >0$ 
with the following property: 
Let $S$ be a distribution on $\R$ satisfying
 \begin{equation}
\label{eqR7.5.11}
S(-t) = \overline{S(t)}, \quad
\supp(S) \sbt  (-b,-a) \cup (a,b), \quad
\sup_{k \in \Z} |\ft{S}(k)| \leq \delta.
\end{equation}
Then there is a 
 real sequence $\Lam = \{\lam_n\}$, $n\in\Z$,
such that 
$|\lam_n - n | \leq \eps$ for all $n$, and
 \begin{equation}
\label{eqR7.5.12}
\text{$\ft{\delta}_\Lam =  \delta_0 + S$
 in $(-b,b)$.}
 \end{equation}
 \end{thm}

\begin{proof}
We choose  an infinitely smooth function  $\Psi$ 
such that $\Psi(-t)=\overline{\Psi(t)}$ for all $t \in \R$,
$\Psi(t) = -1/(2\pi i t)$ in $(-b,-a)\cup(a,b)$,
and $\Psi(t)=0$ for $t \in \R \setminus (-l,l)$.

Let $S$ be a distribution satisfying
\eqref{eqR7.5.11}, then $S \in Y$ and 
  $\|S\|_Y \leq \delta$. Define a new
distribution
$S_1 := S \cdot \Psi$, then also  $S_1 \in Y$.
We have
\[
\|S_1\|_Y \leq M \|S\|_Y, \quad 
M := \|\Psi\|_{A(\T)}.
\]
By \thmref{thmR7.2},  if $\delta$ is small 
enough then there is a sequence $\alpha \in X$,
$\|\alpha\|_X \leq \eps$,
such that the function $F$ defined by \eqref{eqR7.1}
satisfies $\hat F= S_1$ in $(-b,b)$.
It follows that the distributional derivative $F'$
 of the function $F$ satisfies
 \begin{equation}
\label{eqR7.5.15}
\ft{F'}(t) = 2 \pi i t \ft{F}(t) =  2 \pi i t S_1(t) = 
 2 \pi i t \Psi(t) S(t) = - S(t)
\quad \text{in $(-b,b)$},
 \end{equation}
which is true since $2 \pi i t \Psi(t) = -1$ in a neighborhood
of $\supp(S)$.

Let 
$\Lam = \{\lam_n\}$, $n\in\Z$,
be defined by
$\lam_n:=n+\alpha_n$.
Then we have
$|\lam_n - n | \leq \eps$ for all $n$.
It follows from the definition  \eqref{eqR7.1}
of  $F$ that
\[
F' = \sum_{n\in\Z}
(\delta_n-\delta_{\lam_n}) = \delta_{\Z} - \delta_\Lambda,
\]
that is, $\delta_\Lam = \delta_{\Z} - F'$. 
By Poisson's summation formula
$\hat \delta_{\Z}=\delta_{\Z}$, hence we have
\[
\hat \delta_\Lambda =\delta_{\Z}-\ft{{F'}}
= \delta_0 + S
\quad \text{in $(-b,b)$}
\]
due to \eqref{eqR7.5.15}. The proof of \thmref{thmR7.5} is thus concluded.
\end{proof}

\subsection{}
Finally, we observe that
 \thmref{thmR7.15} follows
from \thmref{thmR7.5}.
Indeed, if $S$ is a pseudomeasure on $\R$ such that
$\supp(S) \sbt (a,b)$, then
the distribution $r(S + \widetilde{S})$
satisfies the conditions
\eqref{eqR7.5.11} if 
$r>0$ is sufficiently small.
Hence \thmref{thmR7.5} yields
a sequence $\Lam = \{\lam_n\}$
with the properties as in the statement
of  \thmref{thmR7.15}.


\section{Addendum: A problem of Kolountzakis}
\label{secAD1}

The following question was posed to us
by Kolountzakis: Does there exist a real sequence 
$\Lam = \{\lam_n\}$, $n \in \Z$,
satisfying 
 \begin{equation}
\label{eqR12.1}
A \le \lambda_{n+1} - \lam_n \le B, \quad
n \in \Z,
 \end{equation}
where $A,B>0$ are constants, 
such that $f + \Lam$ is a tiling for some  nonzero 
$f \in L^1(\R)$, but there is no \emph{nonnegative}
$f$ with this property?

The answer turns out to depend on the level 
of the tiling. Suppose first that there is a tiling $f + \Lam$
at some \emph{nonzero} level $w$.
 Then $f$ must have nonzero integral, see
\cite[Lemma 2.3(i)]{KL96}.  In turn this implies 
\cite[Section 4]{KL16} that $\ft{\delta}_\Lam = c \cdot \delta_0$
in some neighborhood $(-\eta, \eta)$
of the origin, where $c$ is a nonzero, positive scalar.
It follows that $f + \Lam$ is a tiling whenever $f$
is a Schwartz function with
$\supp(\ft{f}) \sbt (-\eta, \eta)$.
In particular, there exist tilings
$f + \Lam$ at level one with $f$  nonnegative.

To the contrary, we will construct an example showing
 that the same is not true  if $\Lam$ is only assumed to admit a 
tiling at level zero. We will prove the following
result:

 \begin{thm}
\label{thmR12.2}
There is 
 a real sequence $\Lam = \{\lam_n\}$, $n\in\Z$,
satisfying \eqref{eqR12.1} 
for which
there exist tilings $f + \Lam$
with nonzero
$f \in L^1(\R)$, but  any such a tiling
is necessarily a tiling at  level  zero.
In particular $\Lam$ cannot tile
with any nonnegative (nonzero) $f$.
 \end{thm}

\begin{proof}
Let $a,b$ be two numbers  such that $0 <a<b < \half$.
Let $\psi$ be a smooth even function,
$\psi(t)>0$ on $(-a,a)$, and
$\psi(t)=0$ outside $(-a,a)$.
By \thmref{thmR7.2},
given any $\eps>0$  there is a 
real sequence $\alpha = \{\alpha_n\}$, $n\in\Z$,
satisfying $|\alpha_n| \leq \eps$ for all $n$,
and such that 
$\hat F(t)= r \psi(t)$ in $(-b,b)$ for some $r>0$,
where $F$ is the function  defined by \eqref{eqR7.1}.

Let the sequence
$\Lam = \{\lam_n\}$, $n\in\Z$,
be defined by
$\lam_n:=n+\alpha_n$. Then
 \begin{equation}
\label{eqR12.4}
\ft{\delta}_\Lambda =
 \delta_0 - 2 \pi i  r t \psi(t)
\quad \text{in $(-b,b)$.}
 \end{equation}
This can be shown in the same way as done in
 the proof of \thmref{thmR7.5} above.

In particular we have $\ft{\delta}_\Lambda = 0$
in the open set $G := (-b,-a) \cup (a,b)$, so there exist  nonzero real-valued Schwartz functions $f$ such that $f+\Lam$ is a tiling (at level zero). It suffices to choose $f$ such that $\supp(\ft{f})$ is contained in $G$.

On the other hand, suppose that there is $f \in L^1(\R)$
 such that $f+\Lam$  is a tiling at some \emph{nonzero} level $w$.
Then, as before, this implies that
 $\ft{\delta}_\Lam = c \cdot \delta_0$
in some interval $(-\eta, \eta)$, 
where $c$ is a nonzero (positive) scalar. 
This contradicts \eqref{eqR12.4}, hence no such $f$ exists.
In particular, if $f$ is nonnegative and $f+\Lam$
 is a tiling, then the tiling level must be zero and
 $f$ vanishes a.e.
\end{proof}


\section{Remarks} \label{secR9}

\subsection{}
Let $\Lambda\sbt\R$
 be a discrete set of bounded density. 
If the temperate distribution
$\hat\delta_\Lambda$ is a \emph{measure}
on $\R$, then 
condition \eqref{eqI3.1} is not
only necessary, but also sufficient, for
a function $f \in L^1(\R)$ to tile 
at some level $w$
with the translation set $\Lam$.
In this case the tiling 
level is given by $w = c(\Lam) \ft{f}(0)$,
where $c(\Lam)$ is
the mass that the measure $\ft{\delta}_\Lam$
assigns to the origin
(see \cite[Theorem 2.2]{KL21}).

For example, if $\Lambda$ is a \emph{periodic set}
then $\hat\delta_\Lambda$ is a (pure point) measure,
and $f + \Lam$ is a tiling if and only if
\eqref{eqI3.1} holds. It follows that the set $\Lam$
in \thmref{thmI8.1} is not periodic, nor can it
be represented as a finite union of periodic  sets.

\subsection{}
If $f$ has fast decay, e.g.\ $|f(x)| = o(|x|^{-N})$
as $|x| \to +\infty$ for every $N$, then
$\ft{f}$ is a smooth function and again 
the condition \eqref{eqI3.1} is both
necessary and sufficient for
$f  + \Lam$ to be a tiling at some level $w$.
It follows that the function $f$
in \corref{corI8.2} cannot be chosen to have fast decay.

\subsection{}
\thmref{thmI8.1} also holds in $\R^d$ for every $d\geq 1$.
This can be easily deduced from the one-dimensional result
by taking cartesian products. For example, in $\R^2$
one may take
$F(x,y)=f(x)h(y)$, $G(x,y)=g(x)h(y)$, where
$f$, $g$ are the functions from
\thmref{thmI8.1} and where $h \in L^1(\R)$ is such that
$h + \Z$ is a tiling at level one. 
Then $\ft{F}$, $\ft{G}$ have the 
same set of zeros, but $F$ tiles with 
the translation set $\Lam \times \Z$
while $g$ does not.

\subsection*{Acknowledgement}
The author is grateful to Mihalis Kolountzakis
for motivating discussions and for
posing the problem discussed in
\secref{secAD1}, as well as the problem 
 about positivity
of the functions in \thmref{thmI8.1}.


\end{document}